\documentclass{article}
\usepackage[utf8]{inputenc}
\usepackage[sort]{natbib}
\usepackage{float}
\usepackage{comment}
\usepackage{amsmath,amsthm,amssymb,mathtools,booktabs}
\allowdisplaybreaks
\usepackage{hyperref}
\usepackage{subcaption}
\usepackage[colorinlistoftodos]{todonotes}
\usepackage{cleveref}

\usepackage[a4paper, total={6in, 9in}]{geometry}
\usepackage{bm}
\usepackage[boxed, ruled, linesnumbered,vlined]{algorithm2e}
\usepackage{authblk}

\theoremstyle{plain}
\newtheorem{theorem}{Theorem}[section]
\newtheorem{proposition}[theorem]{Proposition}
\newtheorem{corollary}[theorem]{Corollary}
\newtheorem{lemma}[theorem]{Lemma}
\newtheorem{definition}[theorem]{Definition}

\newtheorem*{observation*}{Observation}
\newtheorem*{Question*}{Question}
\newtheorem*{theorem*}{Theorem}

\DeclareMathOperator{\C}{\mathcal{C}}
\DeclareMathOperator{\F}{\mathcal{F}}

\DeclareMathOperator{\stab}{Stab}

\DeclareMathOperator{\vf}{vf}
\DeclareMathOperator{\fe}{fe}

\DeclareMathOperator{\Sym}{Sym}

\DeclareMathOperator{\Aut}{Aut}

\providecommand{\keywords}[1]{%
  \small	
  \textbf{Keywords:} #1
}

\providecommand{\MSC}[1]{%
  \small	
  \textbf{Mathematics Subject Classification (2020):} #1
}

\newcommand{\define}[1]{\textbf{\emph{#1}}\index{#1}}

\usepackage{amsmath, amssymb, amsfonts} 

\usepackage{ifthen}

\usepackage{tikz}
\usetikzlibrary{calc}
\usetikzlibrary{positioning}
\usetikzlibrary{shapes}
\usetikzlibrary{patterns}

\makeatletter
\edef\texforht{TT\noexpand\fi
  \@ifpackageloaded{tex4ht}
    {\noexpand\iftrue}
    {\noexpand\iffalse}}
\makeatother

\makeatletter
\newif\iftikz@node@phantom
\tikzset{
  phantom/.is if=tikz@node@phantom,
  text/.code=%
    \edef\tikz@temp{#1}%
    \ifx\tikz@temp\tikz@nonetext
      \tikz@node@phantomtrue
    \else
      \tikz@node@phantomfalse
      \let\tikz@textcolor\tikz@temp
    \fi
}
\usepackage{etoolbox}
\patchcmd\tikz@fig@continue{\tikz@node@transformations}{%
  \iftikz@node@phantom
    \setbox\pgfnodeparttextbox\hbox{}
  \fi\tikz@node@transformations}{}{}
\makeatother

\newcommand{\tikzAngleOfLine}{\tikz@AngleOfLine}
\def\tikz@AngleOfLine(#1)(#2)#3{%
  \pgfmathanglebetweenpoints{%
    \pgfpointanchor{#1}{center}}{%
    \pgfpointanchor{#2}{center}}
  \pgfmathsetmacro{#3}{\pgfmathresult}%
}

\tikzset{ 
    vertexNodePlain/.style = {fill=#1, shape=circle, inner sep=0pt, minimum size=2pt, text=none},
    vertexNodePlain/.default=gray,
    vertexPlain/labels/.style = {
        vertexNode/.style={vertexNodePlain=##1},
        vertexLabel/.style={gray}
    },
    vertexPlain/nolabels/.style = {
        vertexNode/.style={vertexNodePlain=##1},
        vertexLabel/.style={text=none}
    },
    vertexPlain/.style = vertexPlain/#1,
    vertexPlain/.default=labels
}
\tikzset{
    vertexNodeNormal/.style = {fill=#1, shape=circle, inner sep=0pt, minimum size=4pt, text=none},
    vertexNodeNormal/.default = blue,
    vertexNormal/labels/.style = {
        vertexNode/.style={vertexNodeNormal=##1},
        vertexLabel/.style={blue}
    },
    vertexNormal/nolabels/.style = {
        vertexNode/.style={vertexNodeNormal=##1},
        vertexLabel/.style={text=none}
    },
    vertexNormal/.style = vertexNormal/#1,
    vertexNormal/.default=labels
}
\tikzset{
    vertexNodeBallShading/pdf/.style = {ball color=#1},
    vertexNodeBallShading/svg/.style = {fill=#1},
    vertexNodeBallShading/.code = {
        \if\texforht
            \tikzset{vertexNodeBallShading/svg=#1!90!black}
        \else
            \tikzset{vertexNodeBallShading/pdf=#1}
        \fi
    },
    vertexNodeBall/.style = {shape=circle, vertexNodeBallShading=#1, inner sep=2pt, outer sep=0pt, minimum size=3pt, font=\tiny},
    vertexNodeBall/.default = white,
    vertexBall/labels/.style = {
        vertexNode/.style={vertexNodeBall=##1, text=black},
        vertexLabel/.style={text=none}
    },
    vertexBall/nolabels/.style = {
        vertexNode/.style={vertexNodeBall=##1, text=none},
        vertexLabel/.style={text=none}
    },
    vertexBall/.style = vertexBall/#1,
    vertexBall/.default=labels
}
\tikzset{ 
    vertexStyle/.style={vertexNormal=#1},
    vertexStyle/.default = labels
}

\newcommand{\vertexLabelR}[4][]{
    \ifthenelse{ \equal{#1}{} }
        { \node[vertexNode] at (#2) {#4}; }
        { \node[vertexNode=#1] at (#2) {#4}; }
    \node[vertexLabel, #3] at (#2) {#4};
}
\newcommand{\vertexLabelA}[4][]{
    \ifthenelse{ \equal{#1}{} }
        { \node[vertexNode] at (#2) {#4}; }
        { \node[vertexNode=#1] at (#2) {#4}; }
    \node[vertexLabel] at (#3) {#4};
}

\newcommand{\edgeLabelColor}{blue!20!white}
\tikzset{
    edgeLineNone/.style = {draw=none},
    edgeLineNone/.default=black,
    edgeNone/labels/.style = {
        edge/.style = {edgeLineNone=##1},
        edgeLabel/.style = {fill=\edgeLabelColor,font=\small}
    },
    edgeNone/nolabels/.style = {
        edge/.style = {edgeLineNone=##1},
        edgeLabel/.style = {text=none}
    },
    edgeNone/.style = edgeNone/#1,
    edgeNone/.default = labels
}
\tikzset{
    edgeLinePlain/.style={line join=round, draw=#1},
    edgeLinePlain/.default=black,
    edgePlain/labels/.style = {
        edge/.style={edgeLinePlain=##1},
        edgeLabel/.style={fill=\edgeLabelColor,font=\small}
    },
    edgePlain/nolabels/.style = {
        edge/.style={edgeLinePlain=##1},
        edgeLabel/.style={text=none}
    },
    edgePlain/.style = edgePlain/#1,
    edgePlain/.default = labels
}
\tikzset{
    edgeLineDouble/.style = {very thin, double=#1, double distance=.8pt, line join=round},
    edgeLineDouble/.default=gray!90!white,
    edgeDouble/labels/.style = {
        edge/.style = {edgeLineDouble=##1},
        edgeLabel/.style = {fill=\edgeLabelColor,font=\small}
    },
    edgeDouble/nolabels/.style = {
        edge/.style = {edgeLineDouble=##1},
        edgeLabel/.style = {text=none}
    },
    edgeDouble/.style = edgeDouble/#1,
    edgeDouble/.default = labels
}
\tikzset{
    edgeStyle/.style = {edgePlain=#1},
    edgeStyle/.default = labels
}

\newcommand{\faceColorY}{yellow!60!white}   
\newcommand{\faceColorB}{blue!60!white}     
\newcommand{\faceColorC}{cyan!60}           
\newcommand{\faceColorR}{red!60!white}      
\newcommand{\faceColorG}{green!60!white}    
\newcommand{\faceColorO}{orange!50!yellow!70!white} 

\newcommand{\faceColor}{\faceColorY}
\newcommand{\faceColorSwap}{\faceColorC}

\tikzset{
    face/.style = {fill=#1},
    face/.default = \faceColor,
    faceY/.style = {face=\faceColorY},
    faceB/.style = {face=\faceColorB},
    faceC/.style = {face=\faceColorC},
    faceR/.style = {face=\faceColorR},
    faceG/.style = {face=\faceColorG},
    faceO/.style = {face=\faceColorO}
}
\tikzset{
    faceStyle/labels/.style = {
        faceLabel/.style = {}
    },
    faceStyle/nolabels/.style = {
        faceLabel/.style = {text=none}
    },
    faceStyle/.style = faceStyle/#1,
    faceStyle/.default = labels
}
\tikzset{ face/.style={fill=#1} }
\tikzset{ faceSwap/.code=
    \ifdefined\swapColors
        \tikzset{face=\faceColorSwap}
    \else
        \tikzset{face=\faceColor}
    \fi
}

\usepackage{hyperref}

\title{A Census of Edge-transitive Surfaces}
\author{Reymond Akpanya\thanks{School of Mathematics and Statistics, The University of Sydney, Carslaw Building F07,
Camperdown NSW 2006, Australia. E-mail: {\tt reymond.akpanya@sydney.edu.au}.}}
\date{}

\begin{document}

\maketitle

\abstract{ 
 In this paper, we study edge-transitive surfaces, i.e.\  triangulated $2$-dimensional manifolds whose automorphism groups act transitively on the edges of these triangulated surfaces. We show that there exist four types of edge-transitive surfaces, splitting up further into a total of five sub-types.  We exploit our theoretical results to compute a census of edge-transitive surfaces with up to $5000$ faces by constructing suitable cycle double covers of edge-transitive cubic graphs.}

\vspace{0.25cm}

\vspace{0.25cm}
\MSC{05E18, 20B25, 05C75 }
\section{Introduction}
One of the main interests of topological graph theory is embedding graphs on surfaces such that the resulting maps visualise prescribed properties of the underlying graphs.
For instance, if a graph $\Gamma$ is highly symmetric, a desirable goal is to embed $\Gamma$ on a surface $S$ via an embedding $\phi:\Gamma\to S$ such that the automorphism group of the map $\phi(\Gamma)$ is as large as possible.
For example, if $\Gamma$ is a $3$-connected planar graph, $\Aut(\Gamma)=\Aut(\phi(\Gamma))$, where $\phi$ denotes the (unique) planar embedding of $\Gamma,$ see \cite{Whitney}. For more studies on graph embeddings, we refer the reader to \cite{tutte_embedding, NPcomplete, richter,MR1133814}.

Here, we investigate strong embeddings (or equivalently cycle double covers) of cubic graphs, i.e.\ embeddings of a cubic graph $\Gamma$ where the faces of a corresponding map are bounded by cycles in $\Gamma$ (see \cite{TopologicalGraphTheory,GraphsOnSurfaces} for more details).  In general, the existence of a CDC for a given cubic graph is still an open problem, formulated in the \emph{cycle double cover conjecture} \cite{szekeres,seymour}.
For surveys related to studies on CDCs of cubic graphs, we refer the reader to \cite{JAEGER19851, zhang}.
If a cubic graph admits a cycle double cover (CDC), then it can be associated to a \emph{simplicial surface} $X$, see \cite{automorphism}. 
Loosely speaking, a simplicial surface is a combinatorial structure that encodes the incidence relations between the vertices, edges and faces of a triangulated 2-manifold. 
Equivalently, a simplicial surface can be interpreted as a map on a surface, where the resulting faces are triangles.
In the case that $\Gamma$ gives rise to a simplicial surface, $\Gamma$ can be viewed as the graph that describes the face-edge incidences of $X$.
We observe that if $\Gamma$ is a cubic graph and $X$ a corresponding simplicial surface forming a regular map (e.g. see \cite{CONDER2001224}), then $\Aut(\Gamma)$ is a group acting transitively on the vertices and edges of $\Gamma$ that can be exploited to construct the CDC corresponding to $X$ as the orbit of a single cycle in $\Gamma$. As a result, $\Aut(X)$ acts transitively on the vertices, edges and faces of $X.$

We aim to construct simplicial surfaces that are slightly less symmetric, namely edge-transitive surfaces.
Here, we call such a surface $X$ edge-transitive, if (1) $\Aut(X)$ acts transitively on the edges of $X$ and (2) the vertex-edge incidences of $X$ yield a simple graph. 
 We aim to achieve this construction by computing suitable CDCs of cubic graphs.
Since it is not clear how to compute a CDC of an arbitrary cubic graph, the main question answered in this work can be summarised as follows:
\begin{Question*}
    How can we efficiently compute CDCs of a cubic graph such that the resulting simplicial surfaces are edge-transitive? 
\end{Question*}
It is easy to see that the cubic graph corresponding to an edge-transitive surface has to be edge-transitive. Hence, we exploit the census of edge-transitive cubic graphs established in \cite{GraphSym,edgetransitiveconder} to conduct our research. In  \cite{akpanya2025censusfacetransitivesurfaces}, the vertex-transitive cubic graphs constructed in \cite{GraphSym,cubicvertextransitive} have been utilised to construct all face-transitive surfaces with up to $1280$ faces.
Inspired by the methods established in the mentioned work, we prove the following theorems.
\begin{theorem*}
    There exist exactly $4$ types of edge-transitive surfaces, splitting up into a total of $5$ subtypes.
\end{theorem*}
\begin{theorem*}
    \label{thm:main}
    There are exactly $2185$ edge-transitive surfaces with up to $5000$ faces.
\end{theorem*}
The computed census of edge-transitive surfaces and algorithms to construct edge-transitive surfaces are available in \cite{edgetransitivesurfacesData}.
Here, we consider the following approach to construct edge-transitive surfaces: 
Let $\Gamma$ be an edge-transitive cubic graph. Our goal is to construct a CDC of $\Gamma$ such that the resulting surface is edge-transitive. For this, we compute a suitable subgroup $H\leq \Aut(\Gamma)$ and exploit $H$ to construct a CDC that yields a desired edge-transitive surface.
We address the conditions and the suitable choice of subgroups, depending on an invariant called the face-edge type (Definition~\ref{definition_vertexedgetype}) in Section~\ref{section:econstruction}. We translate these results into an algorithm to enumerate all edge-transitive surfaces corresponding to a given edge-transitive cubic graph.

Our paper is structured as follows: In 
\Cref{sec:prelim} we present the theoretical background that is essential for this work. Here, we briefly comment on simplicial surfaces and cubic graphs. Note,  a more detailed introduction to simplicial surfaces and related cubic graphs can be found in \cite[Section~2]{akpanya2025censusfacetransitivesurfaces} and also \cite{Akpanya:1016413,Grtzen:992149,automorphism}. In \Cref{section:et} we examine properties of edge-transitive surfaces. In particular, we introduce the face-edge type of an edge-transitive surface and show that there exist exactly four types of edge-transitive surfaces (see \Cref{theorem:invariants2}). In \Cref{section:econstruction} we address the different types of edge-transitive surfaces and hence establish procedures to construct these surfaces from edge-transitive cubic graphs. Lastly, we present our data base of edge-transitive surfaces and discuss our implementations in \Cref{section:etimplementation}.

\section{Preliminaries }\label{sec:prelim}
We start this section by giving some preliminary notes on simplicial surfaces and cubic graphs. 
For this, we introduce some useful notions to handle simplicial complexes and their elements.

If $X$ is a simplicial complex, we define $X_i:=\{x\in X \mid \vert x \vert =i+1\}$.
  For $x\in X$ we set $X_i(x)$ as 
        $X_i(x):=
    \{s \in X_i \mid s\subseteq x\}$ if  $i \leq \vert x\vert$
and as 
    $X_i(x):=
    \{s \in X_i\mid x\subseteq s\}$ if $ i>\vert x\vert.$
  Moreover, we denote the subsets $\{v\}$ of $X$ by $v$ for simplicity.
 This allows us to define simplicial surfaces.
 Here, we refer to a simplicial complex $X$ with $\vert X\vert < \infty$ as a \emph{\textbf{simplicial surface}}, if the following three properties are satisfied:
    (i) $X$ is pure and of dimension $2$,
    (ii) $\vert X_2(e)\vert  = 2$ for all $e \in X_1$,
    (iii) for all $v\in X_0$ there exists an $n>0$ such that $\vert X_2(v)\vert= n$. These $n$ faces can be arranged in a sequence $(F_1,\ldots,F_n)$ such that $\vert F_i\cap F_{i+1}\vert =2$ for $i=1,\ldots, n$ (Subscripts are read modulo $n$). We call $\deg_X(v):=n$ the \textbf{\emph{degree}} of $v$ and the sequence $u(v):=(F_1,\ldots,F_n)$ the \textbf{\emph{umbrella}} of $v$.
The elements in $X_0,X_1$ and $X_2$ are called vertices, edges and faces of $X$, respectively.
    The \textbf{\emph{Euler characteristic}} of $X$ is given by $\chi (X) = |X_0| - |X_1| + |X_2|$. We further denote the \textbf{\emph{automorphism group}} of $X$ by $\Aut(X)$.
 This group acts on $X$ via $\Aut(X)\times X\to X,(\phi,x)\mapsto \phi(x)$. With respect to this action, the \textbf{\emph{orbit}} of $x\in X$ under a subgroup $H\leq \Aut(X)$ is denoted by $x^H$.
We say that $X$ is \textbf{\emph{edge-transitive}}, if $\vert {X_1}^{\Aut(X)}\vert =1$.
For simplicity, we define notions such as connectedness, orientability, etc.\ in the usual way.
Since the connected components of an edge-transitive surface are all isomorphic, we assume the simplicial surfaces in the remainder of this work to be connected.

The focus of this work lies on cubic graphs.
We assume all graphs in this work to be undirected, connected, simple and finite.  Here, a set of cycles of a cubic graph $\Gamma$ is called a \textbf{\emph{cycle double cover}} (CDC), if every edge of $\Gamma$ is contained in exactly two cycles. We call a CDC of $\Gamma$ \textbf{\emph{vertex-faithful}}, if any two cycles of the CDC intersect in at most one edge. Next, we relate simplicial surfaces to cubic graphs.
The \textbf{\emph{face graph}} of a simplicial surface $X$ is the graph $\mathcal{F}(X)=(V,E)$ defined by $V=X_2$ and $E=\{X_2(e)\mid e\in X_1\}.$ Hence, two vertices $F_1,F_2\in V$ are connected in $\mathcal{F}(X)$ if and only if $F_1 \cap F_2\in X_1$. 
Hence, $\mathcal{F}(X)$ forms a cubic graph.
 Note, if a cubic graph $\Gamma$ has a vertex-faithful cycle double cover, then $\Gamma$ is the face graph of a simplicial surface $X$ in the sense of this paper. In this case the cycles of the given CDC correspond to the vertices of the resulting simplicial surface.
We see that $\Aut(X)$ can be embedded into $\Aut(\mathcal{F}(X))$:
As described above, $\Aut(X)$ acts on $X_2$ via $\phi \cdot F = \phi(F)=\{\phi(v_1),\phi(v_2),\phi(v_3)\}$, where $\phi\in \Aut(X)$ and $F=\{v_1,v_2,v_3\}\in X_2$. This action induces a homomorphism $\lambda:\Aut(X) \to \Sym(X_2)$, associating the automorphism $\phi$ to the permutation it effects on $X_2$. 
Since $\phi$ preserves the face-edge incidences of $X$, the permutation $\lambda(\phi)$ can be interpreted as an automorphism of $\F(X)$ and thus $\lambda$ induces a homomorphism
\begin{equation}
\label{eq:lambdaX}
\lambda_X : \Aut(X) \to \Aut(\F(X)).
\end{equation}

We conclude this section by adjusting  the definition of an automorphism-induced $\alpha$-cycle (see \cite[Definition 4.3]{akpanya2025censusfacetransitivesurfaces}) so that it is tailored to our needs.

\begin{definition}
\label{def:alpha}
Let $\Gamma=(V,E)$ be a cubic graph and $F_1,\ldots,F_n\in V$ such that $\{F_i,F_{i+1}\}\in E$ for $i=1,\ldots,n-1$. Furthermore, let $\sigma\in \Aut(\Gamma)$ be an automorphism of order $\ell$ satisfying $\sigma(F_1)=F_n$. We define a \emph{\bf$\sigma$-induced $\alpha$-cycle} as
    \[\alpha(\sigma,F_1,\ldots,F_n):=(\sigma(F_1),\ldots,\sigma(F_{n-1}),\ldots,\sigma^{\ell}(F_1),\ldots,\sigma^{\ell}(F_{n-1}))\]
    if $\vert \{\sigma^i(F_j)\mid i=0,\ldots, \ell-1,\, j=1,\ldots,n-1\} \vert =(n-1)\ell$, and as the empty cycle $()$ otherwise.
\end{definition}

\section{Edge-transitive surfaces}\label{section:et}
 
In this section, we study the action of $\Aut(X)$ on a given edge-transitive surface $X$.
 By definition, $\Aut(X)$ acts transitively on $X_1.$ 
This naturally raises the following question:
\begin{Question*}
How many orbits does 
$\Aut(X)$ induce on ${X_0}$ and ${X_2}$?    
\end{Question*}
 This question is addressed in detail in Propositions~\ref{lemma2:faceorbits} and \ref{lemma:onevertexorbit}.

\begin{proposition}\label{lemma2:faceorbits}
  If $X$ is an edge-transitive surface, then $\vert {X_2}^{\Aut(X)}\vert \leq 2$.
\end{proposition}
\begin{proof}
  Let $e\in X_1$ be an edge with $X_2(e)=\{F_1,F_2\}$.
  We show that an arbitrary face $F\in X_2$ lies in the $\Aut(X)$-orbit of $F_1$ or $F_2.$
  For this purpose, let $e'\in X_1(F)$ be an edge. Since $X$ is edge-transitive, there is an automorphism $\phi \in \Aut(X)$ satisfying $\phi(e)=e'.$ 
  Since $\phi$ respects the incidences of $X$, we know that either $F_1$ or $F_2$ has to be mapped onto $F$ under $\phi$. 
\end{proof}
Hence, if $F$ is a face of an edge-transitive surface $X$, then $\vert F^{\Aut(X)}\vert \in \{\tfrac{1}{2}\vert X_2\vert,\vert X_2\vert\}.$ 
Next, we show that the automorphism group of an edge-transitive surface acts transitively on the corresponding vertices.

\begin{proposition}\label{lemma:onevertexorbit}
  If $X$ is an edge-transitive surface, then $\vert {X_0}^{\Aut(X)}\vert=1.$
\end{proposition}
\begin{proof}
With similar arguments as presented in the proof of Proposition~\ref{lemma2:faceorbits}, we can show that the action of $\Aut(X)$ on $X_0$ yields at most two vertex-orbits. 
So, let us assume that there are exactly two $\Aut(X)$-orbits on $X_0$, namely $V_1$ and $V_2.$
If there exists an edge $e\in X_1$ and an $i\in \{1,2\}$ such that $e$ is incident to two vertices in $V_i$, then  $\vert {X_0}^{\Aut(X)}\vert=1$ follows from the edge-transitivity of $X$. This contradicts our assumption. So, every edge of $X$ is incident to a vertex in $V_1$ and to a vertex in $V_2$.
Now, let $F\in X_2$ be a face with $X_1(F)=\{e_1,e_2,e_3\}$ and $X_0(F)=\{v_1,v_2,v_3\}$, where $v_i\notin X_0(e_i)$ for $i=1,2,3.$
The edges $e_1$ and $e_2$ are both incident to vertices that lie in distinct $\Aut(X)$-orbits. Without loss of generality, we assume $v_3\in V_1$ and $v_1,v_2\in V_2$. This implies that $e_3$ is incident to two vertices in $V_2$, namely $v_1$ and $v_2$. This is a contradiction and hence we conclude the result.

\end{proof}

 In the following, we introduce the face-edge type of a given edge-transitive surface. This invariant will enable us to construct edge-transitive surfaces from given edge-transitive cubic graphs in \Cref{section:econstruction}.
\begin{definition}\label{definition_vertexedgetype}
  Let $X$ be an edge-transitive surface and $S=\stab_{\Aut(X)}(e)$ the stabiliser of an edge $e\in X_1$ in $\Aut(X)$. We define the \emph{\textbf{face-edge type}} $\fe(X)$ of $X$ as 
  \[
  \fe(X):=(\vert {X_2}^{\Aut(X)}\vert, \vert S \vert ).
  \]
\end{definition}
Since the edge-stabilisers of an edge-transitive surface $X$ are all conjugate in $\Aut(X)$, the above invariant is well-defined and not dependant on the edge-choice.
In Section~\ref{section:econstruction}, we make use of the face-edge type to achieve our desired classification of edge-transitive surfaces with at most $5000$ faces. For this purpose, we need to examine the different values taken by the entries of the face-edge type. Because of Proposition~\ref{lemma2:faceorbits} we know that the first entry of a face-edge type is either one or two. The following lemma helps us to further determine the possible values of the second entry of a face-edge type.

\begin{lemma}\label{lem:C2C2}
    Let $X$ be an edge-transitive surface and $e\in X_1$ an edge in $X_1.$ Then $\Aut(X)$ can be embedded into the group $C_2\times C_2.$
\end{lemma}
\begin{proof}
First, let $v_1,v_2\in X_0$ be vertices with $X_0(e)=\{v_1,v_2\}$, and $F_1,F_2\in X_2$ faces satisfying $X_2(e)=\{F_1,F_2\}.$ For every $i=1,2$ let $w_i\in X_0$ be the vertex such that $w_i\in X_0(F_i)\setminus\{v_1,v_2\}.$
Every automorphism $\phi \in \Aut(X)$ that stabilises $e$ has to permute the two vertices $v_1$ and $v_2$, and also $w_1$ and $w_2.$
We know that the automorphism $\phi$ is uniquely identified by the images of $v_1,v_2,w_1,w_2$ under $\phi.$ This means that $\phi$ can be identified with a permutation $$\pi\in \{(),(v_1,v_2),(w_1,w_2),(v_1,v_2)(w_1,w_2)\}=\langle (v_1,v_2),(w_1,w_2)\rangle$$ which implies that $\stab_{\Aut(X)}(e)$ is isomorphic to a subgroup of $C_2\times C_2$.
\end{proof}
This allows us to establish that the face-edge type of an edge-transitive surface $X$ satisfies the following theorem.
\begin{theorem}\label{theorem:invariants2}
Let $X$ be an edge-transitive surface. Then the face-edge type of $X$ satisfies
\[
\fe(X)\in \{(1,2),(1,4),(2,1),(2,2)\}.
\]
\end{theorem}
\begin{proof}
By \Cref{lemma2:faceorbits} and \Cref{lem:C2C2}, we know that the face-edge type of $X$ satisfies $$\fe(X)\in \{(1,1),(1,2),(1,4),(2,1),(2,2),(2,4)\}.$$ Thus, we have to show that $\fe(X)\in \{(1,1),(2,4)\}$ is not possible. We prove this statement by contradiction. 
First, we assume $\fe(X)=(2,4).$ In this case, $\Aut(X)$ has exactly two orbits on $X_2$, where each orbit is of size $\tfrac{\vert X_2\vert }{2}.$
 By the orbit-stabiliser theorem, we then obtain $$\vert\stab_{\Aut(X)}(F)\vert\cdot \tfrac{\vert X_2\vert}{2}  = \vert \Aut(X)\vert =4\cdot \vert X_1\vert =4\cdot \tfrac{3}{2}\cdot \vert X_2 \vert =6\cdot \vert X_2\vert,$$ where $F\in X_2$ is a face. This implies $\vert\stab_{\Aut(X)}(F)\vert=12.$ Since the stabiliser $\stab_{\Aut(X)}(F)$ can be embedded into a symmetric group of order $6$ (see \cite[Lemma 3.7]{akpanya2025censusfacetransitivesurfaces}) this is a contradiction. 

Now, we examine the case $\fe(X)=(1,1).$ This means that  $X$ is face-transitive and all edge-stabilisers are trivial. 
Hence, the orbit-stabiliser theorem implies 
$$\vert\stab_{\Aut(X)}(F)\vert\cdot \vert X_2\vert  =\vert \Aut(X)\vert =1\cdot \vert X_1\vert =1\cdot \tfrac{3}{2}\cdot \vert X_2 \vert = \tfrac{3}{2}\cdot \vert X_2\vert,$$ where $F\in X_2$. 
Thus, $\vert \stab_{\Aut(X)}\vert =\tfrac{3}{2}$ follows, a contradiction. Hence, the result follows.
\end{proof}
\section{Construction of edge-transitive surfaces}
\label{section:econstruction}
Next, we target the construction of edge-transitive surfaces. Here, Theorem~\ref{theorem:invariants2} allows us to analyse edge-transitive surfaces with different face-edge types case by case and propose methods to construct these simplicial surfaces from their face graphs.  
Then, these construction methods are used to approach arbitrary cubic graphs and compute edge-transitive surfaces by providing suitable CDCs.
Since $\lambda_X:\Aut(X)\to \Aut(\F(X))$ is a monomorphism (see Equation~(\ref{eq:lambdaX})), we know that a cubic graph  yielding an edge-transitive surface has to be edge-transitive.
\begin{lemma}\label{lemma:arctransitive}
  The face graph $\mathcal{F}(X)$ of an edge-transitive surface $X$ is edge-transitive. 
\end{lemma}
Loosely speaking, the results of this section can be summarised as follows.
\begin{observation*}
 \label{thm:summaryedge}
 An edge-transitive surface $X$ can be reconstructed by exploiting its face-graph $\mathcal{F}(X)$ and the action of the group $\lambda_X(\Aut(X))\leq \Aut(\F(X))$ on $\mathcal{F}(X)$.   
\end{observation*}
 
An essential tool that allows us to achieve the desired constructions is the notion of a face-colouring of a simplicial surface.

\begin{definition}
    Let $X$ be a simplicial surface. A map $f:X_2\to \{1,2\}$ is called a \define{face-$2$-colouring} of $X$ if every pair of faces $F_1,F_2\in X_2$ with $X_1(F_1)\cap X_1(F_2)\neq \emptyset$ satisfies $f(F_1)\neq f(F_2).$
\end{definition}
We observe that a simplicial surface $X$ has a face-$2$-colouring if and only if $\F(X)$ is bipartite. This provides us with the foundation to conduct the described case-by-case analysis.

\subsection{Edge-transitive surfaces with face-edge type (1,4)}\label{subsection:fe(1,4)}
First, we discuss the structure of an edge-transitive surface with face-edge type $(1,4)$.  
Per definition, such a simplicial surface $X$ is face-transitive. Thus, we obtain the desired description of $X$ by determining the corresponding vertex-face type $\vf(X)$, see \cite[Definition 3.4]{akpanya2025censusfacetransitivesurfaces}. The vertex-face type of a face-transitive surface $X$ is defined similarly to the face-edge type, with the first entry giving $\vert {X_0}^{\Aut(X)}\vert $ and the second entry the order of a face stabiliser. 

\begin{theorem}\label{thm:fe14}
A simplicial surface $X$ is edge-transitive with $\fe(X)=(1,4)$ if and only if $X$ is face-transitive with $\vf(X)=(1,6).$
\end{theorem}
\begin{proof}
If $X$ is edge-transitive with $\fe(X)=(1,4)$, then we know that $X$ is face-transitive. 
We compute the order of a face-stabiliser with the orbit-stabiliser theorem. If $F\in X_2$, then
\[
\vert\stab_{\Aut(X)}(F)\vert\cdot \vert X_2\vert=\vert \Aut(X)\vert =4\cdot\vert X_1\vert=4\cdot \tfrac{3}{2}\cdot \vert X_2\vert= 6\cdot \vert X_2\vert.
\]
Thus, $\vert\stab_{\Aut(X)}(F)\vert=6$ and $\vf(X)=(1,6)$ follow.
Furthermore, if $X$ is face-transitive surface with $\vf(X)=(1,6)$, then the stabiliser of a face $F\in X_2$ is transitive on the edges $X_1(F).$ Hence, $\Aut(X)$ is transitive on $X_1.$ For an edge $e\in X_1$ we obtain 
\[
\vert\stab_{\Aut(X)}(e)\vert\cdot \vert X_1\vert=\vert \Aut(X)\vert =6\cdot\vert X_2\vert=6\cdot \tfrac{2}{3}\cdot \vert X_1\vert= 4\cdot \vert X_1\vert,
\]
 which
 implies $\fe(X)=(1,4).$
 
 \end{proof}
Theorem~\ref{thm:fe14} establishes that if $X$ is an edge-transitive surface with $\fe(X)=(1,4),$ then there exist $\sigma\in H:=\lambda_X(\Aut(X))$ and an edge $\{F_1,F_2\}$ in $\mathcal{F}(X)$ such that the CDC corresponding to $X$ can be computed as the $H$-orbit of the $\alpha$-cycle $\alpha(\sigma,F_1,F_2)^H$ in $\mathcal{F}(X)$, see \cite[Section 4.3]{akpanya2025censusfacetransitivesurfaces} for more details.

\subsection{Edge-transitive surfaces with face-edge type (1,2)}\label{subsection:fe(1,2)}

Next, we examine edge-transitive surfaces with face-edge type $(1,2)$. Again, these simplicial surfaces are face-transitive. Thus, we determine the corresponding vertex-face types in order to obtain our desired description.

\begin{theorem}\label{thm:fe12}
A simplicial surface $X$ is edge-transitive with $\fe(X)=(1,2)$ if and only if $X$ is face-transitive with $\vf(X)=(1,3)$. 
\end{theorem}
\begin{proof}
Let $X$ be edge-transitive with $\fe(X)=(1,2).$ Thus, $X$ is face-transitive and it remains to show that the face-stabilisers of $X$ all have order $3.$ By observing 
\[
\vert\stab_{\Aut(X)}(F)\vert\cdot \vert X_2\vert=\vert \Aut(X)\vert =2\cdot\vert X_1\vert=2\cdot \tfrac{3}{2}\cdot \vert X_2\vert= 3\cdot \vert X_2\vert,
\]
where $F\in X_2$ is an arbitrary face, we conclude $\vf(X)=(1,3)$.
Now, let $X$ be a face-transitive surface with $\vf(X)=(1,3)$. In this case, the stabiliser of a face $F\in X_2$ is cyclic of order $3$ and hence transitive on the edges $X_1(F).$ This implies that $\Aut(X)$ is transitive on $X_1.$ By the orbit-stabiliser theorem, we obtain
\[
\vert\stab_{\Aut(X)}(e)\vert\cdot \vert X_1\vert=\vert \Aut(X)\vert =3\cdot\vert X_2\vert=3\cdot \tfrac{2}{3}\cdot \vert X_1\vert= 2\cdot \vert X_1\vert
\]
for an edge $e\in X_1.$ Thus, $\fe(X)=(1,2)$ follows.
 \end{proof}
Hence, there are exactly two types of edge-transitive surfaces, namely edge-transitive surfaces $X$ with $H=\lambda_X(\Aut(X))$ forming a $(1,3)$-group of $\F(X)$ of type $1$ or type $2$, see \cite[Section 4.5]{akpanya2025censusfacetransitivesurfaces}.
This means that if $H$ is of type $1$, then there exists an edge $\{F_1,F_2\}$ in $\mathcal{F}(X)$ and an automorphism $\sigma\in H$ such that the CDC corresponding to $X$ can be constructed as the $H$-orbit of the $\alpha$-cycle $\alpha(\sigma,F_1,F_2).$ Moreover, if $H$ is of type $2$, then there exist edges $\{F_1,F_2\},\{F_2,F_3\}$ in $\mathcal{F}(X)$ and an automorphism $\sigma'\in H$ such that $X$ arises from the CDC $\alpha(\sigma',F_1,F_2,F_3).$

\subsection{Edge-transitive surfaces with face-edge type (2,2)}\label{subsection:fe(2,2)}
We proceed by focusing on edge-transitive surfaces with face-edge types equal to $(2,2).$ In order to describe the corresponding CDCs in the face graphs of these surfaces, we introduce the definition of a $(2,2)$-group of a given edge-transitive cubic graph.
\begin{definition}\label{def:e(1,4)group}
  Let $\Gamma=(V,E)$ be an edge-transitive cubic graph and $H\leq \Aut(\Gamma)$ such that
  \begin{enumerate}
    \item $H$ acts transitively on $E$ with $\vert  H\vert =2\cdot\vert E\vert,$
    \item the action of $H$ on $V$ yields exactly two $H$-orbits,
    \item for $F_1,F_2,F_3\in V$ with $\{F_1,F_2\},\{F_2,F_3\} \in E,$ there exists $\sigma\in H$ with $\sigma(F_1)=F_3$.
  \end{enumerate}
  If $\alpha(\sigma,F_1,F_2,F_3)^H$  is a vertex-faithful cycle double cover of $\Gamma$, then we say that $H$ is a \emph{\bf${(2,2)}$-group of}\index{${(2,2)}$-group} $\Gamma$. We denote the above cycle double cover by $\C^{{(2,2)}}(H)$.
\end{definition}
The above definition allows us to characterise the desired edge-transitive surfaces.
\begin{theorem}\label{thm:fe21}
 A simplicial surface $X$ is an edge-transitive surface with $\fe(X)=(2,2)$ if and only if $H:=\lambda_X(\Aut(X))$ is a ${(2,2)}$-group of $\F(X)$ and $\{u(v)\mid v\in X_0\}=\C^{{(2,2)}}(H)$, where the umbrellas of $X$ are interpreted as cycles in $\mathcal{F}(X)$. 
\end{theorem}

\begin{proof} 
First, let $X$ be an edge-transitive surface with $\fe(X)=(2,2)$. Thus, $\vert \stab_{\Aut(X)}(e)\vert =2$ for all $e\in X_1$ and $\Aut(X)$ has exactly two orbits on $X_2$. 
Moreover, if $F\in X_2$ is a face, then $\vert \stab_{\Aut(X)}(F)\vert=6$ by the orbit-stabiliser theorem.
In order to prove the result, let $e\in X_1$ be an edge, $\phi\in \stab_{\Aut(X)}(e)$ non-trivial automorphism and $F\in X_2(e)$. 
First, we argue that $\phi$ stabilises $F.$ To prove this statement, we assume the contrary, namely $\phi(F)\neq F.$ Further, let $F'\in X_2$ be an arbitrary face and $e'\in X_1(F')$ an edge. Since $X$ is edge-transitive, there exists an automorphism $\psi\in \Aut(X)$ that maps $e$ onto $e'.$ Note, the non-trivial automorphism $\phi'\in \stab_{\Aut(X)}(e')$ does not stabilise $F'.$ Thus, either $\psi$ or $\phi'\circ \psi$ maps $F$ onto $F'.$ Since $F'$ was arbitrary, $X$ is face-transitive, which contradicts $\fe(X)=(2,2).$

Hence, $\phi$ stabilises $F$. Because $\Aut(X)$ has exactly two orbits on $X_2$, there exists a face-$2$-colouring of $X$ with colour classes corresponding to the $\Aut(X)$-orbits on $X_2$. This helps us to prove that $H:=\lambda_X(\Aut(X))$ is a ${(2,2)}$-group of $\F(X)$ as follows:  
Let $v\in X_0$ be a vertex and $n:=\deg(v)$. Furthermore, for every $1\leq i \leq n$ let $F_i$ be a face and $e_i$ an edge with $v\in X_0(F_i)\cap X_0(e_i).$  More precisely, for every $1\leq i \leq n$ we assume $X_2(e_i)=\{F_i,F_{i+1}\}$, where we read the subscripts modulo~$n,$ see Figure~\ref{fig:ve14} for illustration.

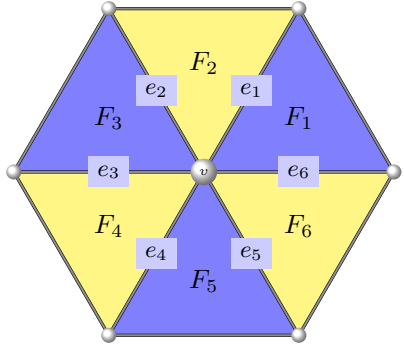
\begin{figure}[H]
    \centering
\scalebox{1.}{
\begin{tikzpicture}[vertexBall, edgeDouble, faceStyle, scale=2.5]

\coordinate (V) at (0., 0.);
\coordinate (V1) at (1., 0.);
\coordinate (V2) at (0.4999999999999999, 0.8660254037844386);
\coordinate (V3) at (-0.5000000000000001, 0.8660254037844386);
\coordinate (V4) at (-1, 0.);
\coordinate (V5) at (-0.5, -0.8660);
\coordinate (V6) at (0.5, -0.8660);

\fill[face=blue!50]  (V) -- (V1) -- (V2) -- cycle;
\fill[face]  (V) -- (V3) -- (V2) -- cycle;
\fill[face=blue!50]  (V) -- (V3) -- (V4) -- cycle;
\fill[face]  (V) -- (V4) -- (V5) -- cycle;
\fill[face=blue!50]  (V) -- (V5) -- (V6) -- cycle;
\fill[face]  (V) -- (V1) -- (V6) -- cycle;

\node[faceLabel] at (barycentric cs:V=1,V1=1,V2=1) {$F_1$};
\node[faceLabel] at (barycentric cs:V=1,V2=1,V3=1) {$F_2$};
\node[faceLabel] at (barycentric cs:V=1,V3=1,V4=1) {$F_3$};
\node[faceLabel] at (barycentric cs:V=1,V4=1,V5=1) {$F_4$};
\node[faceLabel] at (barycentric cs:V=1,V5=1,V6=1) {$F_5$};
\node[faceLabel] at (barycentric cs:V=1,V1=1,V6=1) {$F_6$};

\draw[edge] (V) -- node[edgeLabel] {$e_6$}(V1);
\draw[edge] (V) --node[edgeLabel] {$e_1$}(V2);
\draw[edge] (V) --node[edgeLabel] {$e_2$}(V3);
\draw[edge] (V) --node[edgeLabel] {$e_3$}(V4);
\draw[edge] (V) --node[edgeLabel] {$e_4$}(V5);
\draw[edge] (V) --node[edgeLabel] {$e_5$}(V6);

\draw[edge] (V1) --(V2);
\draw[edge] (V3) --(V2);
\draw[edge] (V3) --(V4);
\draw[edge] (V4) --(V5);
\draw[edge] (V5) --(V6);
\draw[edge] (V1) --(V6);
\vertexLabelR{V1}{left}{$ $}
\vertexLabelR{V2}{left}{$ $}
\vertexLabelR{V3}{left}{$ $}
\vertexLabelR{V4}{left}{$ $}
\vertexLabelR{V5}{left}{$ $}
\vertexLabelR{V6}{left}{$ $}
\vertexLabelR{V}{left}{$v$}

\end{tikzpicture}}
    \caption{The umbrella $u(v)$ of $X$ with $\Aut(X)$-orbits on $X_2$ coloured in blue and yellow}
    \label{fig:ve14}
\end{figure}
Thus, the umbrella of $v$ is given by $u(v)=(F_1,\ldots,F_n)$. We know that $n$ is even and hence we define $k:=\tfrac{n}{2}$.
In the following we aim to construct an automorphism that applies a cyclic shift to the faces that are incident to $v.$ For this, let $\psi_1$ be an automorphism that maps $e_1$ onto $e_3$ and $\psi_2 \in \stab_{\Aut(X)}(e_3)$ non-trivial. Note that $\psi_1$ maps $F_1$ onto $F_3$. Further, we observe that either $\psi_1$ or $\psi_2 \circ \psi_1$ stabilises the vertex $v.$ For simplicity, we assume that $\psi_1$ stabilises $v.$
Because of $\psi_1(F_n)=F_2$ and $\psi_1(F_1)=F_3$, we obtain
\[ 
(F_1,\ldots,F_n)=(\psi_1(F_1),\psi_1(F_2),\ldots,\psi_1^{k}(F_1),\psi_1^{k}(F_2)).\]
Interpreting $(F_1,\ldots,F_n)$ as a cycle of $\F(X)$ and defining the automorphism $\sigma:=\lambda_X(\psi_1)$ yields  $(F_1,\ldots,F_n)=\alpha(\sigma,F_1,F_2,F_3)$ being a $\sigma$-induced $\alpha$-cycle.
Since all the vertices lie in exactly one $\Aut(X)$-orbit (see Proposition~\ref{lemma:onevertexorbit}), the cycle double cover corresponding to $X$ is given by $\alpha(\sigma,F_1,F_2,F_3)^H$ and therefore $H$ forms a ${(2,2)}$-group of $\F(X)$.

Next, let $H:=\lambda_X(\Aut(X))$ be a ${(2,2)}$-group of $\F(X)=(V,E)$ and $e\in X_1$ is an arbitrary edge. The statement $\Aut(X)\cong H$ implies that $X$ is edge-transitive and that the stabiliser $\stab_H(e)$ of $e$ interpreted as an edge of $\F(X)$ is isomorphic to the stabiliser $\stab_{\Aut(X)}(e)$ of $e$ interpreted as an edge of $X$ with $\vert \stab_{\Aut(X)}(e)\vert =2$. 
Moreover, $H$ having exactly two orbits on $V=X_2$ translates directly into $\Aut(X)$ having exactly two orbits on $X_2.$ 
This implies that $X$ is an edge-transitive surface with $\vf(X)=(2,2)$.
\end{proof}

\begin{corollary}\label{corollary:constructionfe22}
    If $\Gamma$ is an edge-transitive cubic graph and $H\leq \Aut(\Gamma)$ a ${(2,2)}$-group of $\Gamma$, then $\Gamma$ is the face graph of an edge-transitive surface.
\end{corollary}

By using Corollary~\ref{corollary:constructionfe22}, we are able to construct a simplicial surface $Y:=X^{(2,2)}$ with $\fe(Y)=(2,2)$.
This simplicial surface $Y$ consists of $36$ vertices (all of vertex-degree $12$), $216$ edges, and $144$ faces. The faces of $Y$ can be found in \Cref{a1}.
We observe that $Y$ is orientable, and satisfies $\chi(Y)=-36$ and
$$\Aut(Y)\cong \Aut(\F(Y))\cong (((C_3\times C_3)\rtimes Q_8)\rtimes C_3)\rtimes C_2.$$
We observe that $Y$ is a minimal edge-transitive surface with $\fe(Y)=(2,2)$ with respect to the number of faces. 

\subsection{Edge-transitive surfaces with face-edge type (2,1)}\label{subsection:fe(2,1)}
Finally, we examine the structure of an edge-transitive surface with face-edge 
type $(2,1)$.  If $X$ is an edge-transitive surface with $\fe(X)=(2,1)$, we prove that the cycles of the CDC in $\mathcal{F}(X)$  corresponding to $X$ form automorphism-induced $\alpha$-cycles. 
Before proceeding further, we give the definition of a ${(2,1)}$-group of a given edge-transitive cubic graph.

\begin{definition}
  \label{def:ef21group}
  Let $\Gamma=(V,E)$ be an edge-transitive cubic graph and $H\leq \Aut(\Gamma)$ such that
  \begin{enumerate}
    \item $H$ acts transitively on $E$ with $\vert  H\vert =1\cdot\vert E\vert,$
    \item the action of $H$ on $V$ yields exactly two orbits,
    \item for $F_1,F_2,F_3\in V$ with $\{F_i,F_{i+1}\} \in E$ for $i=1,2$ there exists $\sigma\in H$ with $\sigma(F_1)=F_3$.
  \end{enumerate}
  If the orbit $\alpha(\sigma,F_1,F_2,F_3)^H$ is a vertex-faithful cycle double cover of $\Gamma$, then we say that $H$ is a \define{${(2,1)}$-group} of $\Gamma$. We denote the above cycle double cover by $\C^{{(2,1)}}(H)$.
\end{definition}
This enables us to formulate the desired description of an edge-transitive surface $X$ satisfying $\fe(X)=(2,1)$.
\begin{theorem}
 A simplicial surface $X$ is an edge-transitive surface with $\fe(X)=(2,1)$ if and only if $H:=\lambda_X(\Aut(X))$ is a ${(2,1)}$-group of $\mathcal{F}(X)$ and $\{u(v)\mid v\in X_0\}=\C^{{(2,1)}}(H)$, where the umbrellas of $X$ are interpreted as cycles in $\mathcal{F}(X)$.
\end{theorem}
\begin{proof} 
Let $X$ be an edge-transitive surface that satisfies $\fe(X)=(2,1)$, i.e.\ $\vert \stab_{\Aut(X)}(e)\vert=1$ for all $e\in X_1$ and $\vert {X_2}^{\Aut(X)}\vert = 2$.
Furthermore, let $v\in X_0$ be a vertex and $n:=\deg(v)$. Since $\Aut(X)$ has exactly two orbits on $X_2$ forming the colour classes of a face-2-colouring, the vertex-degree of $v$ is even and we can define $k:=\tfrac{n}{2}$. Further, for every $1\leq i \leq n$ let $F_i\in X_2$ be a face and $e_i\in X_1$ an edge with $v\in X_0(F_i)\cap X_0(e_i)$. More precisely, for every $1\leq i \leq n$ we set $X_0(F_i)=\{v,v_i,v_{i+1}\}$ and $X_0(e_i)=\{v,v_i\}$, see Figure~\ref{fig:ve21} (Subscripts are read modulo $n$).

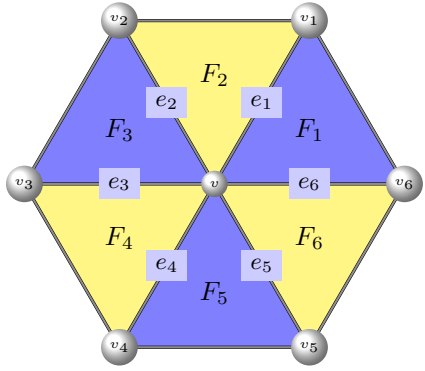
\begin{figure}[H]
    \centering
\scalebox{1.}{
\begin{tikzpicture}[vertexBall, edgeDouble, faceStyle, scale=2.5]

\coordinate (V) at (0., 0.);
\coordinate (V1) at (1., 0.);
\coordinate (V2) at (0.4999999999999999, 0.8660254037844386);
\coordinate (V3) at (-0.5000000000000001, 0.8660254037844386);
\coordinate (V4) at (-1, 0.);
\coordinate (V5) at (-0.5, -0.8660);
\coordinate (V6) at (0.5, -0.8660);

\fill[face=blue!50]  (V) -- (V1) -- (V2) -- cycle;
\fill[face]  (V) -- (V3) -- (V2) -- cycle;
\fill[face=blue!50]  (V) -- (V3) -- (V4) -- cycle;
\fill[face]  (V) -- (V4) -- (V5) -- cycle;
\fill[face=blue!50]  (V) -- (V5) -- (V6) -- cycle;
\fill[face]  (V) -- (V1) -- (V6) -- cycle;

\node[faceLabel] at (barycentric cs:V=1,V1=1,V2=1) {$F_1$};
\node[faceLabel] at (barycentric cs:V=1,V2=1,V3=1) {$F_2$};
\node[faceLabel] at (barycentric cs:V=1,V3=1,V4=1) {$F_3$};
\node[faceLabel] at (barycentric cs:V=1,V4=1,V5=1) {$F_4$};
\node[faceLabel] at (barycentric cs:V=1,V5=1,V6=1) {$F_5$};
\node[faceLabel] at (barycentric cs:V=1,V1=1,V6=1) {$F_6$};

\draw[edge] (V) -- node[edgeLabel] {$e_6$}(V1);
\draw[edge] (V) --node[edgeLabel] {$e_1$}(V2);
\draw[edge] (V) --node[edgeLabel] {$e_2$}(V3);
\draw[edge] (V) --node[edgeLabel] {$e_3$}(V4);
\draw[edge] (V) --node[edgeLabel] {$e_4$}(V5);
\draw[edge] (V) --node[edgeLabel] {$e_5$}(V6);

\draw[edge] (V1) --(V2);
\draw[edge] (V3) --(V2);
\draw[edge] (V3) --(V4);
\draw[edge] (V4) --(V5);
\draw[edge] (V5) --(V6);
\draw[edge] (V1) --(V6);
\vertexLabelR{V1}{left}{$v_6$}
\vertexLabelR{V2}{left}{$v_1$}
\vertexLabelR{V3}{left}{$v_2$}
\vertexLabelR{V4}{left}{$v_3$}
\vertexLabelR{V5}{left}{$v_4$}
\vertexLabelR{V6}{left}{$v_5$}
\vertexLabelR{V}{left}{$v$}

\end{tikzpicture}}
    \caption{The umbrella $u(v)$ of $X$ with $\Aut(X)$-orbits on the faces coloured in blue and yellow}
    \label{fig:ve21}
\end{figure}
Thus, $u(v)=(F_1,\ldots,F_n)$ forms the umbrella of $v$. 
As shown in the proof of Theorem~\ref{thm:fe21}, we seek an automorphism of $X$ that applies a cyclic shift to the faces $F_1,\ldots,F_n$. Let therefore $\phi_1$ be an automorphism that maps $e_1$ onto $e_2$ and $\phi_2 \in \Aut(X)$ an automorphism that maps $e_2$ onto $e_3.$ 
Since $\fe(X)=(2,1)$, we know that $\phi_1(F_2)\neq F_3$ holds. That means that $\phi_1(F_2)=F_2$. Moreover, since the edge-stabiliser of the edge $e\in X_1$ with $X_0(e)=\{v_1,v_2\}$ is trivial, $\phi_1$ cannot stabilise the vertex $v.$ Hence, $\phi_1(v)=v_2,\phi_1(v_1)=v$ and $\phi_1(v_2)=v_1$ hold. Additionally, we know that the edge $e'\in X_1$ with $X_0(e')=\{v_2,v_3\}$ satisfies $\vert \stab_{\Aut(X)}(e')\vert =1.$ Thus, the automorphism $\phi_2$ can also not stabilise $v$, and $\phi_2(v)=v_3,\phi_2(v_3)=v_2$ and $\phi_2(v_2)=v$ follows. This implies that $\phi:=\phi_2\circ \phi_1$ is an automorphism of $X$ satisfying $\phi(v)=v, \phi(v_1)=v_3$ and $\phi(F_1)=F_3$. By observing the incidence structure of $X,$ we obtain $\phi(F_i)=F_{i+2}$ for all $i=1,\ldots,n$, where we read the subscripts modulo $n.$
Thus, the umbrella of $v$ can be written as 
\[ 
(F_1,\ldots,F_n)=(\phi(F_1),\phi(F_2),\ldots,\phi^{k}(F_1),\phi^{k}(F_2)).
\]

If we interpret the umbrella $(F_1,\ldots,F_n)$ as a cycle of $\F(X)$ and define $\sigma:=\lambda_X(\phi)$, we observe that  $(F_1,\ldots,F_n)=\alpha(\sigma,F_1,F_2,F_3)$ is a $\sigma$-induced $\alpha$-cycle.
Since all the vertices lie in exactly one $\Aut(X)$-orbit (see Proposition~\ref{lemma:onevertexorbit}), the cycle double cover corresponding to $X$ is given by $\alpha(\sigma,F_1,F_2,F_3)^H$ and $H$ hence forms a ${(2,1)}$-group of $\F(X)$.

It remains to show that if the subgroup $H$ is a ${(2,1)}$-group of $\F(X)=(V,E)$, then $X$ is edge-transitive with $\fe(X)=(2,1)$.
Therefore, let $H$ be such a group and $e\in X_1$ an arbitrary edge. From $\Aut(X)\cong H$ we deduce that $X$ is edge-transitive and that the stabiliser $\stab_H(e)$ of $e$ interpreted as an edge of $\F(X)$ is isomorphic to the stabiliser of $e$ interpreted as an edge of $X$ such that $\vert \stab_{\Aut(X)}(e)\vert =1$. 
Moreover, $H$ having exactly two orbits on $V$ translates directly into $\Aut(X)$ having exactly two orbits on $X_2.$ Hence, we conclude that $X$ is an edge-transitive surface with $\fe(X)=(2,1)$.
\end{proof}

\begin{corollary}\label{corollary:constructionfe21}
    If $\Gamma$ be an edge-transitive cubic graph and $H\leq \Aut(\Gamma)$ a ${(2,1)}$-group of $\Gamma$, then $\Gamma$ is the face graph of an edge-transitive surface.
\end{corollary}

With Corollary~\ref{corollary:constructionfe21}, we have constructed a simplicial surface $Y:=X^{(2,1)}$ with $\fe(Y)=(2,1)$.
We observe that $Y$ satisfies $(\vert Y_0\vert,\vert Y_1\vert,\vert Y_2\vert)=(28,{16}8,112)$ and $\chi(Y)=-28$.
With GAP we can verify that  
$$\Aut(Y)\cong \Aut(\F(Y))\cong (C_2\times C_2\times C_2 )\rtimes (C_7\rtimes C_3).$$
Note that $Y$ is minimal with the property $\fe(Y)=(2,2)$ with respect to the number of faces.

\section{Notes on implementations}
\label{section:etimplementation}

In this section we describe how we use our theoretical results to compute the census of edge-transitive surfaces with at most $5000$ faces. 
Our implementations and the resulting database of edge-transitive surfaces are available in \cite{edgetransitivesurfacesData}. 
In this paper, the computer algebra systems GAP \cite{GAP4} and Magma \cite{magma} have been exploited to implement the different algorithms to construct and analyse edge-transitive surfaces.
In particular, we make use of the GAP-packages \texttt{simpcomp} \cite{simpcomp}, \texttt{SimplicialSurfaces} \cite{simplicialsurfacegap}, \texttt{GraphSym} \cite{GraphSym} and \texttt{Digraphs} \cite{DeBeule2024aa} to examine and construct edge-transitive surfaces and cubic graphs. We further employ Magma in our studies to speed up the computation of subgroups of the automorphism groups of edge-transitive cubic graphs having prescribed orders.
More precisely, for a given edge-transitive cubic graph $\Gamma=(V,E)$ with corresponding automorphism group $\Aut(\Gamma)$, we apply the following steps to achieve our goal:
\begin{enumerate}
  \item We compute the set $\mathcal{H}$ containing all subgroups $H\leq \Aut(\Gamma)$ that act transitively on the edges of $\Gamma$ and satisfy $\vert H\vert= s \vert V\vert$, where $s\in \{1,2,4\},$
  \item for every $H\in \mathcal{H}$ we compute all edge-transitive surfaces that have $\Gamma$ as a face graph and an automorphism group isomorphic to $H$.
\end{enumerate}
Note that there are exactly $2389$ edge-transitive cubic graphs with at most $5000$ vertices, see \cite{GraphSym,edgetransitiveconder}. We have verified that an edge-transitive cubic graph $\Gamma=(V,E)$ with $\vert V\vert\leq 5000$ satisfies $\vert \Aut(\Gamma)\vert \leq 225792$. Thus, the computation of the set $\mathcal{H}$ is achievable by employing Magma. For each subgroup $H\in \mathcal{H}$, we then use procedures derived from our results in \Cref{section:econstruction} to efficiently compute corresponding edge-transitive surfaces.
The results of our computations are summarised in the following table. In this table, we provide the numbers of different edge-transitive surfaces that are computed with the algorithms described in Section~\ref{section:econstruction}. Here, we denote the set of all edge-transitive surfaces with at most $5000$ faces that have a certain face-edge type $\fe(X),$ where $X$ is an edge-transitive surface, by $M({\fe(X)}).$  
\begin{table}[H]
    \centering
    \begin{tabular}{|*{8}{c|}}
        \hline
        \bm{$\fe(X)$} & $\bm{\vert {M({\fe(X)})}\vert} $ & \textbf{Type 1} & \textbf{Type 2} \\ & & & \vspace{-0.4cm} \\ 
        \hline
  \textbf{(1,4)} & $790$ & $\times$& $\times$ \\
        \textbf{(1,2)} & $1040$ & $958$ & $82$\\
        \textbf{(2,2)}& $119$& $\times$& $\times$ \\
        \textbf{(2,1)}& $236$ & $\times$& $\times$ \\
        \hline
    \end{tabular}
    \caption{Numbers of edge-transitive surfaces with at most $5000$ faces with respect to the different face-edge types}

\end{table}

We write "$\times$" if the type of edge-transitive surfaces described by the corresponding entry is not defined.
Thus, there are exactly $790+958+82+119+236=2185$ edge-transitive surfaces with at most $5000$ faces, up to isomorphism. Out of these edge-transitive surfaces exactly $2002$ are orientable and $183$ are non-orientable. 
In \cite{edgetransitivesurfacesData,Akpanya:1016413}, we give more detailed information on the different numbers of edge-transitive surfaces with respect to their Euler characteristics and their face-edge types. 
Note, in these tables we write $(f,s).i$ if the corresponding column or entry labelled by $(f,s).i$ contains information about edge-transitive surfaces $X$ with face-edge type $(f,s),$ where the groups $\lambda_X(\Aut(X))$ form $(f,s)$-groups of type $i$ of $\F(X).$ Similarly, we make use of the notation $(f,s).$ The following table contains the numbers $n(\fe(X))$ that describe the number of faces of minimal examples of edge-transitive surfaces with face-edge types $\fe(X)$ (with respect to the number of faces).

\begin{table}[h!]
    \centering
    \begin{tabular}{|*{11}{c|}}
        \hline
          \bm{$\fe(X)$}& \textbf{(1,4)} & \textbf{(1,2).1} & \textbf{(1,2).2} &  \textbf{(2,2)}& \textbf{(2,1)} \\
        & & &  &&\vspace{-0.4cm}\\  
        \hline
  $\bm{n({\fe(X)})} $ & $4$& $14$& $144$ & $144$ & $112$\\
        \hline
    \end{tabular}
    \caption{Minimal number of faces of an edge-transitive surface with a certain face-edge type}
   
\end{table}

We have implemented several tests to ensure that the derived census of edge-transitive surfaces contained in \cite{edgetransitivesurfacesData} is correct. In particular, we have conducted the following tests:
\begin{enumerate}
    \item For every edge-transitive cubic graph $\Gamma$ with at most $56$ vertices, we have constructed all edge-transitive surfaces by computing all possible CDCs of $\Gamma$ with the help of the implementations in \cite{simplicialsurfacegap}. We then verified that applying our algorithms to $\Gamma$ produced the same set of edge-transitive surfaces, up to isomorphism.
    \item For every edge-transitive cubic graph $\Gamma=(V,E)$ with $\vert \Aut(\Gamma)\vert\leq 4000$, we computed all subgroups $H\leq \Aut(\Gamma)$ of order $\vert H\vert =s\vert V\vert $, where $s=1,2,4$, using the GAP-function \texttt{ConjugacyClassesSubgroups}, see \cite{GAP4}. We then checked whether the candidate subgroups computed via our algorithm, which is based on computing chains of maximal subgroups is equal to the above set of subgroups, up to conjugacy.
    \item For every edge-transitive cubic graph $\Gamma=(V,E)$ with $\vert V\vert \leq 5000$ we have computed a copy $\Gamma'$ of $\Gamma$ by randomly relabelling the vertices of $\Gamma$ with the labels $1,\ldots,\vert V\vert $. We then checked whether the set of edge-transitive surfaces with $\Gamma$ as face graph is equal to the set of edge-transitive surfaces with $\Gamma'$ as face graph, up to isomorphism.
\end{enumerate}
 We are aware of the fact that errors may still occur despite these efforts of testing.

\subsection*{Acknowledgements}

The author acknowledges funding by the Deutsche Forschungsgemeinschaft (DFG, German Research Foundation) in the framework of the Collaborative Research Centre CRC/TRR 280 ``Design Strategies for Material-Minimized Carbon Reinforced Concrete Structures – Principles of a New Approach to Construction'' (project ID 417002380). Furthermore, the author was
supported by a grant from the Simons Foundation (SFI-MPS-Infrastructure-00008650). This paper builds upon results that first appeared in the author’s PhD dissertation submitted to RWTH Aachen University, see \cite[Section 6]{Akpanya:1016413}. The author thanks Alice C.\ Niemeyer, Daniel Robertz, Jonathan Spreer, and Meike Weiß for their helpful comments on earlier versions of this paper.

\appendix

\section{Examples}

\subsection{Set of faces of the surface \texorpdfstring{$X^{(2,2)}$}{X(2,2)}}
\label{a1}
\begin{align*}
\{&\{{34},{35},{36}\},\{{32},{33},{35}\},\{{31},{32},{34}\},\{{29},{30},{34}\},\{{28},{29},{36}\},\{{26},{27},{35}\},\\
    &\{{25},{26},{36}\},\{{23},{24},{35}\},\{{22},{23},{33}\},\{{20},{21},{32}\},
\{{20},{28},{33}\},\{{18},{19},{34}\},\\
&\{{17},{24},{34}\},\{{15},{16},{32}\},\{{18},{22},{31}\},\{{14},{17},{30}\},\{{15},{25},{31}\},\{{12},{13},{29}\},\\
&\{{12},{27},{30}\},\{{19},{23},{36}\},
\{{17},{19},{35}\},\{{18},{24},{36}\},\{{11},{16},{35}\},\{{10},{20},{29}\},\\
&\{{14},{23},{28}\},\{9,{17},{27}\},\{9,{18},{25}\},\{8,{11},{24}\},\{7,{12},{26}\},\{{11},{15},{26}\},\\
&\{5,6,{23}\},\{5,{21},{24}\},\{{15},{27},{33}\},\{{11},{27},{32}\},\{{16},{26},{33}\},\{{10},{21},{34}\},\\
&\{{10},{30},{32}\},\{{17},{18},{23}\},\{8,{15},{22}\},\{{11},{13},{21}\},
\{{13},{26},{28}\},\{6,{10},{19}\},
\\
&\{7,{10},{16}\},\{5,{14},{20}\},\{4,8,{18}\},\{2,3,{17}\},\{4,9,{15}\},\{4,{16},{19}\},\\&
\{2,{13},{24}\},\{{20},{30},{31}\},
\{{21},{29},{31}\},\{7,{13},{36}\},\{7,{25},{29}\},\{6,{20},{22}\},\\
&\{3,7,{19}\},\{2,9,{12}\},\{1,2,{11}\},\{1,6,{16}\},\{{12},{25},{28}\},\{6,{14},{33}\},\\
&
\{3,{12},{14}\},\{2,{14},{27}\},\{1,5,{13}\},\{1,3,{10}\},\{1,8,{21}\},\{3,9,{30}\},
\\&
\{5,{22},{28}\},\{8,9,{31}\},\{2,5,8\},\{1,4,7\},
\{3,4,6\},\{4,{22},{25}\},\\
&\{{32},{34},{35}\},\{{29},{34},{36}\},\{{26},{35},{36}\},\{{23},{33},{35}\},\{{20},{32},{33}\},\{{18},{31},{34}\},\\
&\{{17},{30},{34}\},\{{15},{31},{32}\},
\{{12},{29},{30}\},\{{23},{28},{36}\},\{{17},{27},{35}\},\{{18},{25},{36}\},\\
&\{{11},{24},{35}\},\{{20},{28},{29}\},\{{12},{26},{27}\},\{{15},{25},{26}\},\{5,{23},{24}\},\{{15},{22},{33}\},\\
&
\{{11},{21},{32}\},\{{26},{28},{33}\},\{{10},{19},{34}\},\{{21},{24},{34}\},\{{10},{16},{32}\},\{{18},{22},{23}\},\\
&
\{5,{20},{21}\},\{4,{18},{19}\},\{2,{17},{24}\},\{4,{15},{16}\},
\{{20},{22},{31}\},\{{14},{20},{30}\},\\
&
\{{25},{29},{31}\},\{{13},{21},{29}\},\{{27},{30},{32}\},\{7,{19},{36}\},\{{16},{19},{35}\},\{{13},{24},{36}\},\\
&
\{7,{10},{29}\},\{{14},{17},{23}\},
\{2,{12},{13}\},\{6,{19},{23}\},\{3,{17},{19}\},\{8,{18},{24}\},
\\
&\{1,{11},{16}\},\{6,{10},{20}\},\{{12},{14},{28}\},\{9,{15},{27}\},\{9,{12},{25}\},\{7,{16},{26}\},\\
&
\{{11},{13},{26}\},\{{14},{27},{33}\},\{6,{16},{33}\},\{9,{17},{18}\},\{8,{11},{15}\},\{3,7,{12}\},\\
&\{1,5,6\},\{2,{11},{27}\},\{1,{10},{21}\},\{3,{10},{30}\},
\{5,8,{22}\},\{5,{13},{28}\},\\
&\{9,{30},{31}\},\{8,{21},{31}\},\{2,5,{14}\},\{1,4,8\},\{1,2,3\},\{3,4,9\},\\
&\{1,7,{13}\},\{4,7,{25}\},
\{4,6,{22}\},\{{22},{25},{28}\},\{2,8,9\},\{3,6,{14}\}\}.
\end{align*}

\subsection{Set of faces of the surface \texorpdfstring{$X^{(2,1)}$}{X(2,1)}}
\begin{align*}
\{& \{ 26, {27}, {28} \}, \{ {24}, {25}, {27} \}, \{ {22}, {23}, {26} \}, \{ {21}, {24}, {26} \}, \{ {19}, {20}, {28} \}, \{ {17}, {18}, {27} \},\\&
\{ {16}, {22}, {28} \}, \{ {14}, {15}, {25} \}, 
  \{ {13}, {18}, {26} \}, \{ {11}, {12}, {24} \}, \{ {11}, {19}, {27} \}, \{ {{10}}, {17}, {25} \}, \\&
  \{ 8, 9, {23} \}, \{ 6, 7, {21} \}, \{ {13}, {17}, {28} \}, \{ 5, {12}, {27} \}, 
  \{ 4, {21}, {22} \}, \{ {14}, {22}, {24} \}, \\&
  \{ {12}, {18}, {23} \}, \{ 3, {11}, {21} \}, \{ 2, 5, {20} \}, \{ 1, 9, {18} \}, \{ 3, 6, {16} \}, \{ 5, {11}, {25} \},\\& \{ 4, {14}, {26} \}, 
  \{ 4, {23}, {24} \}, \{ {10}, {16}, {19} \}, \{ 8, {19}, {22} \}, \{ 6, {17}, {26} \}, \{ 4, {13}, {20} \},\\& \{ 1, {16}, {21} \}, \{ 2, 4, {15} \}, \{ 2, 7, {13} \}, \{ 1, 8, {12} \}, 
  \{ 1, 6, {11} \}, \{ {14}, {21}, {23} \},\\&
  \{ 8, {10}, {28} \}, \{ 6, {13}, {27} \}, \{ {10}, {20}, {22} \}, \{ 3, {10}, {14} \}, \{ 5, {19}, {24} \}, \{ 9, {14}, {17} \},\\
  & \{ 6, {18}, {28} \}, 
  \{ 5, {13}, {15} \}, \{ 7, {11}, {16} \}, \{ 9, {10}, {15} \}, \{ 3, {15}, {17} \}, \{ 4, 5, 7 \},\\& \{ 8, {16}, {20} \}, \{ 3, 9, {25} \}, \{ {12}, {19}, {25} \}, \{ 2, 8, {18} \}, 
  \{ 2, 9, {12} \}, \{ 1, 3, 7 \}, \\&\{ 1, 2, {23} \}, \{ 7, {15}, {20} \}, \{ {24}, {26}, {27} \}, \{ {22}, {26}, {28} \}, \{ {19}, {27}, {28} \}, \{ {17}, {25}, {27} \},\\& \{ {14}, {24}, {25} \}, 
  \{ {18}, {23}, {26} \}, \{ {11}, {21}, {24} \}, \{ 8, {22}, {23} \}, \{ 6, {21}, {26} \}, \{ {13}, {20}, {28} \},\\&
  \{ {12}, {18}, {27} \}, \{ {16}, {21}, {22} \}, \{ 5, {19}, {20} \}, 
  \{ 9, {17}, {18} \}, \{ 6, {16}, {28} \}, \{ 5, {15}, {25} \},\\&
  \{ 4, {13}, {26} \}, \{ {12}, {23}, {24} \}, \{ {11}, {16}, {19} \}, \{ {10}, {17}, {28} \}, \{ 4, {14}, {15} \}, 
  \{ 2, {13}, {18} \},\\& \{ 1, {11}, {12} \}, \{ 6, {11}, {27} \}, \{ {10}, {19}, {25} \}, \{ 9, {14}, {23} \}, \{ 4, 7, {21} \}, \{ 5, {13}, {27} \}, \\&\{ 4, {20}, {22} \}, \{ {10}, {14}, {22} \}, 
  \{ 3, {11}, {25} \}, \{ 8, 9, {10} \}, \{ 6, 7, {13} \}, \{ {13}, {15}, {17} \},\\& \{ 2, 5, {12} \}, \{ {19}, {22}, {24} \}, \{ 3, {14}, {21} \}, \{ 2, 8, {20} \}, \{ 1, 6, {18} \}, 
  \{ 3, {10}, {16} \},\\& \{ {14}, {17}, {26} \}, \{ 4, 5, {24} \}, \{ 8, {12}, {19} \}, \{ 3, 6, {17} \}, \{ 1, {21}, {23} \}, \{ 1, 3, 9 \}, \\&\{ 5, 7, {11} \}, \{ 2, 4, {23} \}, 
  \{ 1, 8, {16} \}, \{ 2, 9, {15} \}, \{ 8, {18}, {28} \}, \{ 7, {16}, {20} \},\\& \{ 1, 2, 7 \}, \{ {10}, {15}, {20} \}, \{ 9, {12}, {25} \}, \{ 3,7, {15} \} \}.
\end{align*}

\bibliographystyle{abbrv}
\bibliography{main.bib}
\addcontentsline{toc}{section}{References}

\end{document}